\documentclass[12pt, a4paper]{amsart}
\usepackage{url,graphicx}
\usepackage{multirow, xcolor}
\usepackage{fullpage}
\usepackage{amsfonts,amscd,amssymb,amsmath,amsthm,mathrsfs,multirow,lscape,xfrac}
\usepackage{pbox,lipsum, stmaryrd,url, epstopdf,float, enumerate}
\usepackage{array}
\usepackage{xspace, comment}
\usepackage[all,cmtip]{xy}
\usepackage[english]{babel}
\usepackage{color}

\usepackage{ucs}
\usepackage[utf8]{inputenc}
\usepackage{dsfont}
\usepackage[T1]{fontenc} 
\usepackage{lmodern}
\usepackage[thinlines]{easytable}

\usepackage{phaistos}
\usepackage{microtype}

\usepackage{enumitem}
\usepackage{scalerel}
\usepackage{todonotes}
\usepackage{stackengine,wasysym}
\usepackage{todonotes}
\usepackage[colorlinks,linkcolor=blue,anchorcolor=blue,citecolor=blue,backref=page]{hyperref} % Clickable refs and citations
\usepackage{cleveref}
\usepackage[english]{babel}

\newtheorem{theorem}{Theorem}[section]
\newtheorem{cor}[theorem]{Corollary}
\newtheorem{lemma}[theorem]{Lemma}

\newtheorem{definition}[theorem]{Definition}
\newtheorem{remark}[theorem]{Remark}

\newcommand{\CC}{\mathbb{C}}
\newcommand{\QQ}{\mathbb{Q}}
\newcommand{\RR}{\mathbb{R}}
\newcommand{\ZZ}{\mathbb{Z}}

\newcommand{\N}{\mathscr{N}}

\def\tareesidedbox#1{%
\setbox0=\hbox{$\vphantom{\zeta^{(k)}}#1$}%
\dimen0=\wd0
\advance\dimen0 by3pt
\rlap{\hbox{\vrule height11pt width.4pt depth2pt
\kern-.4pt
\vrule height11.4pt width\dimen0 depth-11pt
\kern-.4pt
\vrule height11pt width.4pt depth2pt}}%
\hbox to\dimen0{\hss$\vphantom{\zeta^{(k)}}#1$\hss}}
\def\house#1{\tareesidedbox{#1}}

\title[On the classification of small cyclotomic integers]{On the classification of small cyclotomic integers}
\author{Jitendra Bajpai}
\address{Department of Mathematics, Christian-Albrechts University of Kiel, 24118 Kiel, Germany and Einstein Institute of Mathematics, Edmond J. Safra Campus, The Hebrew University of Jerusalem, Jerusalem 9190401, Israel}
\email{jitendra@math.uni-kiel.de}

\author{Srijan Das}
\address{Department of Mathematics, Indian Institute of Science Education and Research Pune, Maharashtra, India}
\email{das.srijan@students.iiserpune.ac.in}

\author{Kiran S. Kedlaya}
\address{Department of Mathematics, University of California San Diego, 9500 Gilman Drive \#0112, La Jolla, CA, 92122, USA}
\email{kedlaya@ucsd.edu}

\author{Jorge Mello}
\address{Department of Mathematics and Statistics, Oakland University, Auburn Hills, Michigan, 48309, USA}
\email{jorgedemellojr@oakland.edu}

\subjclass[2020]{Primary 11R18; secondary 11R06, 11Y40}
\keywords{Cassels's theorem, cyclotomic integers}
\begin{document}
\date{\today}

\begin{abstract}
We give a general classification theorem for cyclotomic algebraic integers with all complex absolute values bounded by a fixed constant $c$, modeled on the theorem of Cassels which treats the case $c = \sqrt{5}$ up to finitely many exceptions. As a corollary, we establish that the range of the function taking a cyclotomic integer to its maximum complex absolute value is a well-ordered (but not closed) subset of the real numbers. We also formulate analogous statements for algebraic numbers in the maximal cyclotomic extension of a fixed number field. The proofs combine a result of Loxton, which bounds the number of roots of unity in the shortest additive representation of a cyclotomic integer in terms of the maximum complex absolute value, 
with an equidistribution theorem of Bilu et al. for Galois orbits of torsion points on algebraic tori.
\end{abstract}

\maketitle

\tableofcontents

%--------------------------------------------------------------------------
%	INTRODUCTION
%-------------------------------------------------------------------------
\section{Introduction}

The problem of classifying cyclotomic algebraic integers which are small (i.e., which have small absolute value under any embedding into $\CC$) was originally raised by Robinson \cite{robinson} and subsequently studied in detail by numerous authors, notably including Cassels \cite{cassels}. In this paper, we describe a general form for the classification of small cyclotomic integers, modeled on a slightly modified formulation of \cite[Theorem~I]{cassels}. 

To give this formulation, we must introduce some notation.
For any number field $K$, we denote by $K^{\mathrm{cycl}}$ its cyclotomic closure (i.e., the compositum with the maximal abelian extension of $\mathbb{Q}$).
For a given algebraic integer $\alpha$, we define the \emph{house} of $\alpha$, denoted $\house{\alpha}$, to be the maximum of $|\beta|$ as $\beta$ varies over conjugates of $\alpha$ in $\CC$; this is evidently invariant under multiplication by a root of unity.
It is often more convenient to work with the quantity $\house{\alpha}^2$; following \cite{BDKLLLM}, we refer to this as the \emph{castle} of $\alpha$.

In this notation,  \cite[Theorem~I]{cassels} can be stated as follows.

\begin{theorem}[Cassels] \label{T:cassels}
For $\alpha$ a cyclotomic algebraic integer in $\mathbb{Q}^{\mathrm{cycl}}$, $\house{\alpha}^2 < 5.01$ if and only if $\alpha$ is a root of unity times one of the following.
\begin{enumerate}
\item $1+\zeta$ for some root of unity $\zeta$.
\item $1 + \zeta- \zeta^{-1}$ for some root of unity $\zeta$.
\item $(\zeta_5+\zeta_5^4) + (\zeta_5^2 + \zeta_5^3)\zeta$ for some root of unity $\zeta$.
\item A member of a certain effectively computable finite set.
An explicit set of this form is identified in  \cite{BDKLLLM}.
\end{enumerate}
\end{theorem}

In this paper, we consider what happens when the constant $5.01$ is replaced by an arbitrary positive real number. To state our first result, we let $\mu \subset \mathbb{Q}^{\mathrm{cycl}}$ denote the group of roots of unity.

\begin{theorem} \label{T:general classification rational case}
 For any real number $C>0$, the set of cyclotomic integers $\alpha$ with $\house{\alpha}^2 \leq C$
 can be written in the form
 \[
 \{P_i(\zeta_1,\dots,\zeta_{n_i}) \colon i \in I; \zeta_1, \dots, \zeta_{n_i} \in \mu\}
 \]
 for some finite family $\{P_i(z_1,\dots,z_{n_i})\}_{i \in I}$ of polynomials $P_i$ with cyclotomic integer coefficients.
\end{theorem}

We also establish a ``semicontinuity'' theorem generalizing the following result of Robinson--Wurtz \cite{robinson-wurtz}: the function $\alpha \mapsto \house{\alpha}^2$ on cyclotomic integers takes no values in the range $(5, 5.01)$.

\begin{theorem} \label{T:semicontinuity rational case}
For any real number $C>0$, there exists $\epsilon> 0$ with the property that there is no cyclotomic integer $\alpha$ with $C < \house{\alpha}^2 < C + \epsilon$.
\end{theorem}

An equivalent formulation of \Cref{T:semicontinuity rational case} is that the range of the function
$\alpha \mapsto \house{\alpha}^2$ on cyclotomic integers is a \emph{well-ordered} subset of $\RR$: every infinite nonincreasing sequence valued in this range is eventually constant. For comparison, we note that this range is not a \emph{closed} subset of $\RR$: whereas every element of the range is itself a (totally real, totally positive) cyclotomic integer, Calegari \cite{calegari-house} observed that 
\[
\frac{97 + 26 \sqrt{13}}{27} \approx 7.064604
\]
is a limit point of the range which is visibly not a cyclotomic integer. (Finding the smallest such limit point will be the topic of subsequent work.)

%\kiran{add references to Bilu, Laurent to bibTeX}\jitendra{just added them}
The proofs of \Cref{T:general classification rational case} and \Cref{T:semicontinuity rational case} depend on two key inputs: a theorem of Loxton \cite{loxton} showing that a cyclotomic integer of bounded castle can be written as a sum of a bounded number of roots of unity, and an equidistribution theorem of Yuan \cite{yuan} concerning Galois orbits of torsion points on tori, generalizing an earlier result of Bilu \cite{bilu}.
(We note a strong resemblance with the argument given in \cite[\S 5]{bilu} to recover Laurent's theorem classifying algebraic subvarieties of tori in which the torsion points are Zariski dense \cite{laurent}.)

Using similar ingredients, we also obtain similar results about algebraic integers in $K^{\mathrm{cycl}}$ for any number field $K$. These generalizations are naturally stated and proved in terms of $S$-integers in $K^{\mathrm{cycl}}$ where $S$ is an arbitrary finite set of places of $K$; see \Cref{thm:general classification} and \Cref{T:uniform gap}.

Although we have not included any statements about effectivity in our main results, for any fixed bound the resulting classification is effectively computable. We include some remarks about this point in \S\ref{sec:remarks on effectivity}.

\begin{comment}
\begin{itemize}
    \item Try to compute $c$ in~\Cref{T:loxton}.
    \item Generalize the current argument over number field/cyclotomic field.
\end{itemize}
\end{comment}

\section{Loxton's theorem}

We state the key theorem of Loxton \cite{loxton} and use it to take a first step in the direction of proving
Theorem~\ref{T:general classification rational case}.
As in \cite{BDKLLLM} we define the \emph{minimal weight} of a cyclotomic integer $\alpha$, denoted $\N(\alpha)$, to be the smallest nonnegative integer $n$ such that $\alpha$ can be written as a sum of $n$ roots of unity (not necessarily distinct). By the triangle inequality, we have $\N(\alpha) \geq \house{\alpha}$; Loxton's theorem is a bound in the opposite direction.

\begin{theorem}[Loxton]\label{T:loxton}
There exists a nondecreasing function $L \colon \RR \to \RR$ with the property that for every cyclotomic integer $\alpha$, 
\[
\N(\alpha) \leq L(\house{\alpha}).
\]
We refer to any such function $L$ as a \emph{Loxton function}.
\end{theorem}

This formally promotes to a comparable statement over a general number field, as observed by Dvornicich--Zannier \cite{DZ}.
\begin{theorem}\label{loxton+}\cite[Theorem L]{DZ} For any number field $K$,
there exist a number $B=B(K)$ and a finite set $E=E(K) \subset K$ with $\# E \leq [K: \mathbb{Q}]$ such that any algebraic integer $\alpha \in K^{\mathrm{cycl}}$ can be written as $\alpha=\sum_{i=1}^b c_i \zeta_i$  where  $c_i \in E$, the $ \zeta_i$'s are roots of unity, $b \leq \# E \cdot  L(B\house{\alpha})$, and $L\colon \mathbb{R} \rightarrow \mathbb{R}$ is any Loxton function in the sense of \Cref{T:loxton}. The constants $B$ and $E$ can be computed in terms of a basis for $K/\mathbb{Q}$.
\end{theorem}

\begin{remark} \label{rem:loxton general case}
When $K = \QQ$, \Cref{loxton+} specializes back to \Cref{T:loxton} by taking $B = 1$, $E = \{1\}$. 
In this case the situation is more precise than in the general case of \Cref{loxton+}: not only does every algebraic integer $\alpha \in \QQ^{\mathrm{cycl}}$ have a representation as $\sum_{i=1}^b c_i \zeta_i$, but conversely every element of $\QQ^{\mathrm{cycl}}$ which can be represented this way is actually an algebraic integer. One cannot achieve this for a general number field $K$ because the ring of $K$-cyclotomic integers is not guaranteed to be a finitely generated module over the ring of $\QQ$-cyclotomic integers.

By contrast, let $L$ be a Galois extension of $\QQ$ such that every nonarchimedean place of $L$ has divisible value group (e.g., the maximal \emph{solvable} extension of $\QQ$). Then \cite[Theorem~7(iv)]{lyu} implies that the ring of integers in $L$ is a universally Japanese Pr\"ufer domain; consequently, the ring of integers in the compositum $KL$ is a finitely generated module over the ring of integers of $L$.

We will address this point in \S\ref{sec:nonarchimedean} by putting archimedean and nonarchimedean places on a common footing.
\end{remark}

%Loxton~\cite{loxton} established the following lower bound on $\calM(\alpha)$ in terms of $\N(\alpha)$.

%\begin{theorem}[Loxton]\label{T:loxton}
%For any $k > \log 2$, there exists an effectively computable (in terms of $k$) constant $c$ such that for all nonzero cyclotomic integers $\alpha \in \mathbb{Q}^c$,
%\[
%\calM(\alpha) \geq c n \exp (-k \log n/ \log \log n), \qquad n := \N(\alpha).
%\]
%On the other hand, no such constant exists for $k = \log 2$.
%\end{theorem}
%In particular, if $\alpha=\sum_{i=1}^b \zeta_i$ is a sum of roots of unity, then one can choose the roots of unity $\zeta_i$ so that $b \leq L(\house{\alpha})$ where $L: \mathbb{R}_+ \rightarrow \mathbb{R}_+$ is a suitable specified function that we can take to satisfy $L(x) \ll_\epsilon x^{2+\epsilon}.$ We refer to such function $L(x)$ (so that the above is true) as a \textit{Loxton function}. We will make use of the following extension of Loxton's Theorem to algebraic integers contained in a cyclotomic extension of a given number field.

We apply Loxton's theorem as follows.

\begin{definition}  \label{def:covering family}
For a real number $C > 0$ and a number field $K$, define a \emph{$C$-covering family over $K$} to be a finite set of Laurent polynomials  $\{P_i(z_1,\dots,z_{n_i})\}_{i \in I}$ with coefficients in $K^{\mathrm{cycl}}$ having the following property: 
for every algebraic integer $\alpha$ of $K^{\mathrm{cycl}}$  with $\house{\alpha}^2 \leq C$, $\alpha$ can be written as
$P_i(\zeta_1,\dots,\zeta_{n_i})$ for some index $i \in I$ and some $\zeta_1,\dots,\zeta_{n_i} \in \mu$.

When $K = \QQ$, we further require that each $P_i$ has cyclotomic \emph{integer} coefficients.
This restriction is unreasonable for general $K$ in light of \Cref{rem:loxton general case}.
\end{definition}

\begin{lemma} \label{L:covering family exists}
For every $C > 0$ and every number field $K$, there exists a $C$-covering family over $K$.
\end{lemma}
\begin{proof}
Let $\alpha \in K^{\mathrm{cycl}}$ be an algebraic integer with castle $\house{\alpha}^2 \leq C$, which implies $\house{\alpha} \leq \sqrt{C}$. By \Cref{loxton+}, there exist a constant $B = B(K)$ and a finite set $E = E(K) \subset K$ such that $\alpha$ can be written as a linear combination of roots of unity:
\[
\alpha = \sum_{i=1}^b c_i \zeta_i,
\]
where each coefficient $c_i \in E$, the $\zeta_i$ are roots of unity, and the number of terms $b$ satisfies the upper bound
\[
b \leq (\# E) \cdot L(B\house{\alpha}) \leq (\# E) \cdot L(B\sqrt{C}),
\]
for a fixed Loxton function $L$. Define $N = \lfloor (\# E) \cdot L(B\sqrt{C}) \rfloor$.
Consequently, the finite family of linear polynomials
\[
\mathcal{F} = \left\{ P(z_1, \dots, z_b) = \sum_{i=1}^b c_i z_i \;\middle|\; 1 \leq b \leq N, \; c_i \in E \right\}.
\]
is a $C$-covering family over $K$.
\end{proof}

\section{Equidistribution of Galois orbits on tori}

We next introduce an equidistribution theorem for Galois orbits in algebraic tori. This requires a few definitions in order to formulate the result.

\begin{definition}%[Strict Sequence]
A sequence $\{\alpha_k\}$ of points in $(\overline{\QQ}^*)^N$ is \emph{strict} if any proper algebraic subgroup of $(\overline{\mathbb{Q}}^*)^N$ contains $\alpha_k$ for only finitely many values of $k$.
\end{definition}

In connection with the previous definition, we make the following observation.
\begin{lemma} \label{L:Subgroup-Ctbility}
Let $T \cong \mathbb{G}_m^N$ be an algebraic torus. The set of proper closed algebraic subgroups of $T$ is countable.
\end{lemma}
\begin{proof}
Let $H \subsetneq T$ be a proper closed algebraic subgroup, and let $H^\circ$ denote its identity component. 
Since $H$ is an algebraic variety, it has finitely many connected components, and so $H$ is a finite disjoint union of cosets of $H^\circ$. Since $T$ is a commutative algebraic group consisting of semisimple elements, any connected closed subgroup is again a torus; in particular, $H^\circ$ is a subtorus of $T$. 

We show that there are only countably many such subgroups $H$. First, any subtorus $H^\circ \subseteq T$ is uniquely determined by its annihilator in the character lattice $X(T) \cong \ZZ^N$. Since $\ZZ^N$ is a finitely generated module over the noetherian ring $\ZZ$, it has only countably many subgroups. Thus there are only countably many possibilities for $H^\circ$.

Next, fix such an $H^\circ$ and consider the possible cosets. The group $H/H^\circ$ is finite, so each coset has finite order. Moreover, $H^\circ$, being a torus over an algebraically closed field, is divisible. Therefore any coset can be represented by a torsion point of $T(\overline{\QQ})$: if $x\in T(\overline{\QQ})$ is such that $xH^\circ$ has order $m$, then $x^m \in H^\circ$, and by divisibility there exists $y \in H^\circ(\overline{\QQ})$ such that $y^m = x^m$; then $\zeta = xy^{-1}$ satisfies $\zeta^m = 1$ and represents the same coset.

Finally, the torsion points of $T(\overline{\QQ})$ correspond to elements of $\mu^N$. Since $\mu$ is countable, it follows that the set of torsion points of $T$ is countable. Hence there are only countably many possibilities for the cosets, and therefore only countably many such subgroups $H$.
\end{proof}

\begin{cor} \label{cor:finite union of torsion cosets}
Let $\{\alpha_k\}_k$ be a sequence of points in $(\overline{\QQ}^*)^N$ admitting no infinite subsequence which is strict. 
Then $\{\alpha_k\}_k$
is contained in some finite union of subsets of $(\overline{\QQ}^*)^N$ of the form $gH(\overline{\QQ})$ where $g$ is a torsion point and $H$ is a proper subtorus of $(\overline{\QQ}^*)^N$.
\end{cor}
\begin{proof}
By Lemma \ref{L:Subgroup-Ctbility}, there are only countably many proper algebraic subgroups of $\mathbb{G}_m^{N}$. Enumerate them as $\{ H_v\}_{v \ge 1}$. We claim that the sequence $\{\alpha_k\}_k$ is contained in finitely many proper algebraic subgroups. Suppose, by way of contradiction, that this is not the case. Then for each $v \ge 1$, we can choose $\alpha_{k_v}$ such that
\[
\alpha_{k_v} \in H_v(\overline{\QQ}) \quad \text{but} \quad \alpha_{k_v} \notin H_s(\overline{\QQ}) \ \text{for all } s < v.
\]
By construction, the subsequence $\{\alpha_{k_v}\}_v$ is strict, contradicting our assumption on the original sequence $\{\alpha_k\}_k$. 

Hence $\{\alpha_k\}_k$ is contained in finitely many proper algebraic subgroups. 
As in the proof of \Cref{L:Subgroup-Ctbility},
each of these subgroups has identity component equal to a proper subtorus of $(\overline{\QQ}^*)^N$
and each coset is generated by a torsion point. This proves the claim.
\end{proof}

\begin{definition}%[Absolute Logarithmic Height]
Given $\alpha = (\alpha^{(1)}, \dots, \alpha^{(N)}) \in (\overline{\QQ}^*)^N$, the \emph{absolute logarithmic height} of $\alpha$ is defined as
\begin{equation}
h(\alpha) := \frac{1}{[K:\QQ]} \sum_{v} [K_v:\QQ_v] \max\left(0, \log|\alpha^{(1)}|_v, \dots, \log|\alpha^{(N)}|_v\right)
\end{equation}
where $K$ is a number field containing the points $\alpha^{(1)}, \dots, \alpha^{(N)}$ and the summation extends over all valuations of $K$, normalized so their restrictions to $\QQ$ define the usual infinite or $p$-adic valuations. (Note that $h(\alpha)$ does not depend on the choice of $K$.)

A sequence $\{\alpha_k\}$ of points in $(\overline{\QQ}^*)^N$ is \emph{small} if $h(\alpha_k) \to 0$ as $k \to \infty$.
\end{definition}

\begin{definition}
A sequence $\{\alpha_k\}$ of points in $(\overline{\QQ}^*)^N$ is \emph{generic} if no infinite subsequence is contained in a proper closed subvariety of $\mathbb{G}_m^N$.
\end{definition}

\begin{definition}%[Galois Probability Measure]
Fix a place $v$ of $\overline{\QQ}$ which may be archimedean or nonarchimedean,
and let $\CC_v$ be the completion of $\overline{\QQ}$ with respect to $v$.
For $\alpha \in (\overline{\QQ}^*)^N$, the probability measure $\overline{\delta}_\alpha$ on $(\CC_v^*)^N$ is defined by
\begin{equation}
\overline{\delta}_\alpha = \frac{1}{[K:\QQ]} \sum_{\sigma: K \hookrightarrow \CC_v} \delta_{\sigma(\alpha)}
\end{equation}
where $K$ is a number field containing the coordinates of $\alpha$,
the summation is extended over all distinct embeddings of $K$ into $\CC_v$, 
and $\delta_{\sigma(\alpha)}$ denotes the Dirac measure at $\sigma(\alpha)$.
(Again, the definition does not depend on the choice of $K$.)
\end{definition}

%\begin{definition}%[Limit Measure $\nu$]
%Let $\nu$ be the probability measure on $(\CC^*)^N$ supported on the unit polycircle $|z_1| = \dots = |z_N| = 1$, where %\end{definition}

We now state an equidistribution theorem, which was proved by Bilu for the torus $(\overline{\QQ}^*)^N$; by Szpiro, Ullmo, and Zhang for abelian varieties; and has since been proven and generalized under different settings. The version we state below is a consequence of a general result by Yuan over number fields.

\begin{theorem} \label{yuan_equidistribution} Let $\{\alpha_k\}$ be an infinite sequence of points in $(\overline{\QQ}^*)^N$ which is small and generic. 
\begin{itemize}
    \item[(a)]
If $v$ is an archimedean place of $\overline{\QQ}$, then the
probability measures $\overline{\delta}_{\alpha_k}$ on $(\CC_v^*)^N$ 
converge weakly to the Haar measure on the unit polycircle $|z_1| = \dots = |z_N| = 1$.
\item[(b)]    
If $v$ is a nonarchimedean place of $\overline{\QQ}$, let $\mathbb{P}^{1,{\mathrm{an}}}(\CC_v)$ denote the analytic projective line in the sense of Berkovich, and let $\zeta_G \in \mathbb{P}^{1,{\mathrm{an}}}(\CC_v)$ be the Gauss point (i.e., the unique boundary point of the disc $|z| \leq 1$).
Then the probability measures $\overline{\delta}_{\alpha_k}$ considered on $(\mathbb{P}^{1,{\mathrm{an}}}(\CC_v))^N$ converge weakly to the Dirac measure $\delta_{\zeta_G^N}$ of $\zeta_G^N = (\zeta_G, \dots, \zeta_G)$.
\end{itemize}
\end{theorem}
\begin{proof}
Part (a) is the statement of \cite[Theorem, page 605]{yuan}, which is a special case of \cite[Theorem~3.7]{yuan}.
The corresponding special case of \cite[Theorem~3.7]{yuan} in the nonarchimedean case yields (b);
more precisely one gets the analogous conclusion for the analytification of $(\mathbb{P}^1)^N$,
which contains the $N$-th power of $\mathbb{P}^{1,{\mathrm{an}}}(\CC_v)$ as a closed subspace.
\end{proof}

\begin{cor} \label{apply yuan}
Let $\mu \subset \overline{\QQ}^*$ be the group of roots of unity.
Then the conclusion of \Cref{yuan_equidistribution} holds for any strict sequence  $\{\alpha_k\}$ in $\mu^N$.
\end{cor}
\begin{proof}
We must check that the given sequence is small and generic. The small condition is automatic because every $\alpha \in \mu^N$ satisfies $h(\alpha) = 0$. The generic condition holds because by Laurent's theorem \cite{laurent}, any infinite sequence of points of $\mu^N$ lying in a proper closed subvariety of $\mathbb{G}_m^N$ contains an infinite subsequence of points lying in a proper closed subgroup of $\mathbb{G}_m^N$.
\end{proof}

\section{Classification}

Using equidistribution, we next prove our main results about cyclotomic integers, \Cref{T:general classification rational case} and \Cref{T:semicontinuity rational case}. We also obtain some partial results about $K$-cyclotomic integers for a general number field $K$; we will improve upon these results in the next section.

We start with a refinement of \Cref{def:covering family}.
\begin{definition} \label{def:classifying family}
For any number field $K$, define a \emph{$C$-classifying family for $K$} to be a $C$-covering family for which every algebraic integer in $K^{\mathrm{cycl}}$ of the form $P_i(\zeta_1,\dots,\zeta_{n_i})$ for some index $i \in I$ and some roots of unity $\zeta_1,\dots,\zeta_{n_i}$ has castle at most $C$.
As in \Cref{def:covering family}, when $K \neq \QQ$ we do not require each $P_i$ to have algebraic integer coefficients.
\end{definition}

\begin{remark}
In \cite{cassels} and \cite{BDKLLLM}, the classification of small cyclotomic integers is formulated modulo an equivalence relation generated by rotation (i.e., multiplication by a root of unity) and the action of $\mathrm{Gal}(\overline{\QQ}/\QQ)$. Here it will be more convenient to omit this; one can easily recover similar conclusions which are stated modulo this equivalence.
\end{remark}

\begin{definition}
For $\sigma \colon \overline{\QQ} \to \CC$ a field embedding and $P$ a Laurent polynomial with coefficients in $\overline{\QQ}$, let $\sigma(P)$ be the Laurent polynomial over $\CC$ obtained by applying $\sigma$ to each coefficient of $P$.    
\end{definition}

\begin{lemma} \label{L:reduce covering family for K}
Let $K$ be a number field and let $\{P_i\}_{i \in I}$ be a $C$-covering family over $K$.
%(where each $P_i$ is a polynomial with coefficients in $K^{\mathrm{cycl}}$). 
Suppose that $i \in I$ is an index for which
\begin{equation} \label{eq:maximize over circles}
\max\{|\sigma(P_i)(z_1,\dots,z_{n_i})|^2 \colon \sigma\colon K^{\mathrm{cycl}} \hookrightarrow \CC; z_1, \dots, z_{n_i} \in S^1\} > C.    
\end{equation}
\begin{enumerate}
    \item[(a)]
    An infinite sequence $\{\alpha_k\}_k$ of distinct tuples $\alpha_k = (\zeta_1^{(k)},\dots,\zeta_{n_i}^{(k)})$ of roots of unity for which $\house{P_i(\alpha_k)}^2 \leq C$
    cannot be strict in $\mathbb{G}_m^{n_i}$.

    \item[(b)]
There exists a finite set of polynomials $\{Q_j(z_1,\dots,z_{n_j})\}_{j \in J}$ with coefficients in $K^{\mathrm{cycl}}$ such that $n_j < n_i$ for each $j \in J$ and 
\[
\{P_{i'}\}_{i' \in I \setminus \{i\}} \bigcup \{Q_j\}_{j \in J}
\]
is a $C$-covering family over $K$.
\end{enumerate}
%Then there exists a finite set of polynomials $\{Q_j(z_1,\dots,z_{n_j})\}_{j \in J}$ with coefficients in $K^{\mathrm{cycl}}$ such that $n_j < n_i$ for each $j \in J$ and 
%\[
%\{P_{i'}\}_{i' \in I \setminus \{i\}} \bigcup \{Q_j\}_{j \in J}
%\]
%is a $C$-covering family over $K$.
\end{lemma}

\begin{proof}
Because $P_i$ has a finite number of coefficients belonging to $K^{\mathrm{cycl}}$, these coefficients generate a finite extension of $K$. Therefore, they are contained in $L = K(\zeta_m)$ for some primitive $m$-th root of unity $\zeta_m$. We can write $P_i(z_1, \dots, z_{n_i}) = \tilde{P}_i(\zeta_m, z_1, \dots, z_{n_i})$, where $\tilde{P}_i$ is a polynomial with coefficients strictly in $K$.

We fix an embedding $\overline{K} \hookrightarrow \mathbb{C}$, allowing us to identify points in $\mathbb{P}^n(\overline{K})$ with complex points. Let $d = [L:K]$. We denote $d$ distinct $K$-embeddings of $L$ into $\overline{K}$ by $\tau_1, \dots, \tau_d$. Evaluating these on the generator $\zeta_m$ yields $d$ distinct primitive $m$-th roots of unity, $w_j = \tau_j(\zeta_m)$. For any $\sigma \in \text{Gal}(\overline{K}/K)$, $\sigma|_L$ coincides with exactly one of the embeddings $\tau_j$.

Because the coefficients of $\tilde{P}_i$ lie in $K$, applying any $\sigma \in \text{Gal}(\overline{K}/K)$ to $P_i$ yields $\sigma(P_i)(Z) = \tilde{P}_i(\sigma(\zeta_m), Z)$. Since $\sigma(\zeta_m) \in \{w_1, \dots, w_d\}$, hypothesis \eqref{eq:maximize over circles} is equivalent to stating that there exists some integer $j \in \{1, \dots, d\}$ such that $\max_{Z \in (S^1)^{n_i}} |\tilde{P}_i(w_j, Z)|^2 > C$. Up to reordering the embeddings $\tau_j$, we may assume without loss of generality that this maximum occurs at $w_1$.

We prove (a) by way of contradiction.
We lift the sequence to $\beta_k = (\zeta_m, \alpha_k) \in \mathbb{G}_m^{n_i+1}(\overline{K})$. Let $G = S^1 \times (S^1)^{n_i}$ be the ambient compact abelian group. Let $K_k=K(\beta_k)$ and $\text{Aut}(K_k)$ denotes the set of complex embeddings of $K_k$. The uniform probability measure supported on the absolute Galois orbit $O(\beta_k)$ is:
$$\bar{\delta}_{\beta_k} = \frac{1}{[K_k:K]} \sum_{\sigma \in \text{Aut}(K_k)} \delta_{\sigma(\beta_k)}$$

Let $\nu_{n_i}$ be the normalized Haar measure on $(S^1)^{n_i}$. We define the regular Borel probability measure $\nu_W$ on $G$ as:
$$\nu_W = \frac{1}{d} \sum_{j=1}^d (\delta_{w_j} \times \nu_{n_i})$$
which is supported on the disjoint union $W = \bigsqcup_{j=1}^d (\{w_j\} \times (S^1)^{n_i})$.

To establish the weak convergence $\bar{\delta}_{\beta_k} \xrightarrow{w} \nu_W$ on $G$, we use the fact that the continuous characters of $G$ span a dense subspace of $C(G)$ (the set of real continuous functions of $G$). The characters are of the form $\chi_{a,b}(w, z) = w^a z^b = w^a z_1^{b_1} \cdots z_{n_i}^{b_{n_i}}$ for $(a,b) \in \mathbb{Z} \times \mathbb{Z}^{n_i}$. It suffices to show that for every character, $\lim_{k \to \infty} \int_G \chi_{a,b} \, d\bar{\delta}_{\beta_k} = \int_G \chi_{a,b} \, d\nu_W$.

If $b = 0$, the character is $\chi_{a,0}(w,z) = w^a$.  The integral with respect to the orbit measure is:
$$\int_G w^a \, d\bar{\delta}_{\beta_k} = \frac{1}{[K_k : K]} \sum_{\sigma \in \text{Aut}(K_k)} \sigma(\zeta_m)^a$$
Because $L \subseteq K_k$, every embedding $\tau_j$ of $L$ extends to exactly $[K_k : L]$ distinct automorphisms in $\text{Aut}(K_k)$. Partitioning the sum by these restrictions allows the degree of intermediate field extension $K_k/L$ to factor out:
$$\int_G w^a \, d\bar{\delta}_{\beta_k} = \frac{1}{[K_k : L] d} \sum_{j=1}^d [K_k : L] \tau_j(\zeta_m)^a = \frac{1}{d} \sum_{j=1}^d w_j^a$$
On the RHS via Fubini's theorem we have:
\[
\int_G w^a \, d\nu_W = \frac{1}{d} \sum_{j=1}^{d} \left( \int_{S^1} w^a \, d\delta_{w_j} \right) \left( \int_{(S^1)^{n_i}} 1 \, d\nu_{n_i} \right) = \frac{1}{d} \sum_{j=1}^{d} w_j^a
\]
and so the desired equality holds.

If $b \neq 0$, the character is $\chi_{a,b}(w,z) = w^a z^b$. The integral is:
$$\int_G w^a z^b \, d\bar{\delta}_{\beta_k} = \frac{1}{[K_k : K]} \sum_{\sigma \in \text{Aut}(K_k)} \sigma(\zeta_m^a \alpha_k^b)$$
Define the scalar $\gamma_k \in \mathbb{G}_m^1(\overline{K})$ as $\gamma_k = \zeta_m^a \alpha_k^b = \zeta_m^a (\zeta_1^{(k)})^{b_1} \cdots (\zeta_{n_i}^{(k)})^{b_{n_i}}$. Then the above sum is precisely $\int_{S^1} x \, d\bar{\delta}_{\gamma_k}$, where $\bar{\delta}_{\gamma_k}$ is the Galois orbit measure of $\gamma_k$.

As $\gamma_k$ is a root of unity, its canonical logarithmic height is $h(\gamma_k) = 0$, satisfying smallness. Moreover, for sake of contradiction, assume that $\{\gamma_k\}_k$ is not generic. Since proper closed subvarieties of $\mathbb{G}_m^1$ are finite, there exists a root of unity $c$ and a subsequence $\{k_v\}$ such that $\zeta_m^a \alpha_{k_v}^b = c$. Thus, $\alpha_{k_v}$ satisfies $z^{Nb} = 1$ for some integer $N \ge 1$. Since $b \neq 0$, this defines a proper algebraic subgroup of $\mathbb{G}_m^{n_i}$ containing $\{\alpha_{k_v}\}_v$, contradicting the assumption that $\{\alpha_k\}_k$ is strict. Thus we may apply \Cref{yuan_equidistribution}(a) to deduce that the Galois orbit measures $\bar{\delta}_{\gamma_k}$ converge weakly to the canonical probability measure on $\mathbb{G}_m^1(\mathbb{C})$, which is the normalized Haar measure $\nu_{\text{Haar}}$ on $S^1$. Consequently:
$$\lim_{k \to \infty} \int_G w^a z^b \, d\bar{\delta}_{\beta_k} = \lim_{k \to \infty} \int_{S^1} x \, d\bar{\delta}_{\gamma_k} = \int_{S^1} x \, d\nu_{\text{Haar}} = 0$$
%\jorge{Why is this equal 0 and why is it needed at this point? Sorry if I am missing something silly or already discussed here}
Concurrently, since $b \neq 0$, the orthogonality of characters on $(S^1)^{n_i}$ implies $\int_{(S^1)^{n_i}} z^b \, d\nu_{n_i} = 0$. Consequently:
\[
\int_G w^a z^b \, d\nu_W = \frac{1}{d} \sum_{j=1}^{d} w_j^a \left( \int_{(S^1)^{n_i}} z^b \, d\nu_{n_i} \right) = 0.
\]
This establishes the weak convergence $\bar{\delta}_{\beta_k} \xrightarrow{w} \nu_W$ on $G$ under the assumption of strictness of $\{\alpha_k\}$.

Now fix $t>0$, and choose a continuous function $h_t:\mathbb{R}\to\mathbb{R}$ such that
\[
h_t(x)=
\begin{cases}
1 & x \leq C,\\
\in (0,1) & C < x < C+t,\\
0 & x \geq C+t.
\end{cases}
\] 
Define the continuous function $g_{t,i} = h_t \circ |\tilde{P}_i|^2$ on $(S^1)^{n_{i}+1}$.
For any Galois conjugate $Y \in O(\beta_k)$, the condition $\house{P_i(\alpha_k)}^2 \leq C$ ensures that $|\tilde{P}_i(Y)|^2 \leq C$. This forces $g_{t,i}(Y) = 1$ for every $Y \in O(\beta_k)$, yielding:
$$
\int_G g_{t,i} \,d\bar{\delta}_{\beta_k} = 1.
$$
Conversely, we established that $\max_{Z} |\tilde{P}_i(w_1, Z)|^2 > C$. Therefore, $g_{t,i}$ is strictly less than $1$ on a nonempty open subset of the connected component $\{w_1\} \times (S^1)^{n_i}$. Because $\nu_W$ assigns uniform positive mass $1/d$ to this component, we conclude:
$$
\int_G g_{t,i} \, d\nu_W < 1.
$$
This contradicts the weak convergence, completing the proof of (a).

To prove (b), we may assume that $n_i > 0$, as otherwise we may simply omit $P_i$ and obtoain another $C$-covering family. Let $\{\alpha_k\}_k$ be an enumeration of all of the tuples $\alpha_k = (\zeta_1^{(k)},\dots,\zeta_{n_i}^{(k)})$ of roots of unity for which $\house{P_i(\alpha_k)}^2 \leq C$.
By (a), this sequence admits no strict subsequence. Using \Cref{cor:finite union of torsion cosets}, we deduce that $\{\alpha_k\}_k$ is contained in the union of finitely many subvarieties of the form $g\tilde{H}$ where $g$ is a torsion point of $\mathbb{G}_m^{n_i}$ and $\tilde{H}$ is a proper subtorus.
 
Fix one such torsion subvariety $g\tilde{H}$. By replacing $g\tilde{H}$ with a larger proper torsion subvariety if necessary, we may assume $\dim(g\tilde{H}) = n_i - 1$. Thus there exists a parametrization
\[
\psi\colon \mathbb{G}_m^{n_i-1}(\overline{K}) \to \mathbb{G}_m^{n_i}(\overline{K})
\]
of the form $\psi(y_1,\dots,y_{n_i-1}) = (x_1,\dots,x_{n_i})$, where each $x_i$ is a monomial in the variables $y_1,\dots,y_{n_i-1}$ with coefficient a root of unity, and $\mathrm{Im}(\psi) = g\tilde{H}$.

Restricting $P_i$ to $g\tilde{H}$ yields $P_i|_{g\tilde{H}} = P_i \circ \psi$, which is a polynomial in $n_i - 1$ variables. Because the coefficients of $P_i$ lie in $K^{\mathrm{cycl}}$, and the parametrization $\psi$ introduces only roots of unity (which trivially belong to $K^{\mathrm{cycl}}$), the resulting polynomial has coefficients in $K^{\mathrm{cycl}}$;
moreover, if $K = \QQ$ then the resulting polynomial has algebraic integer coefficients because $P_i$ was assumed to do so. Replacing $P_i$ with the collection of polynomials obtained by varying the torsion subvariety $g\tilde{H}$ yields a new $C$-covering family over $K$ of the desired form. 
\end{proof}

\begin{lemma} \label{L:well ordering}
Let $S$ be the set of sequences $(m_0, m_1, \dots)$ of nonnegative integers with finite support.
Define the following total ordering on $S$:
\[
(m_0, m_1, \dots) < (m'_0, m'_1, \dots) \Longleftrightarrow \mbox{for some $j \geq 0$, } m_j < m'_j \mbox{ and } m_k = m'_k  \mbox{ for all } k > j.
\]
Then this is a \emph{well ordering} on $S$: there are no infinite strictly decreasing sequences. 
\end{lemma}
\begin{proof}
Suppose to the contrary that $S$ contains an infinite strictly decreasing sequence. By comparing with the first term in the sequence, we see that there exists $n \geq 0$ such that
all of the terms $(m_0, m_1, \dots)$ in the sequence  satisfy $m_i = 0$ for $i>n$.
We thus reduce to the assertion that the set of $n$-tuples of nonnegative integers, equipped with the lexicographic ordering, is well-ordered.
For this, see for example \cite[\S 2.2, Proposition 4]{CoxLittleOshea}.
\end{proof}

We now obtain a result that includes \Cref{T:general classification rational case} (by taking $K = \QQ$).
\begin{theorem} \label{T:general classification}
For any number field $K$ and any real number $C>0$, there exists a $C$-classifying family for $K$.
\end{theorem}
\begin{proof}
    By Lemma~\ref{L:covering family exists}, there exists a $C$-covering family. 
    For each such family, we can form a sequence $(m_0, m_1, \dots)$ where $m_j$ counts the number of indices $i$ with $n_i = j$. By Lemma~\ref{L:well ordering}, there is a minimal such sequence; choose a $C$-covering family
    realizing this minimum. The minimality means that the conclusion of
    Lemma~\ref{L:reduce covering family for K} cannot hold for any index $i$,
    and thus the hypothesis of Lemma~\ref{L:reduce covering family for K} cannot hold for any index $i$. This means that the family in question is in fact a $C$-classifying family.
\end{proof}

By a slight modification, we obtain a result that includes \Cref{T:semicontinuity rational case}.
\begin{theorem} \label{T:semicontinuity general}
For any number field $K$ and any real number $C>0$, there exists $\epsilon> 0$ with the property that there is no $K$-cyclotomic integer with castle in the range $(C, C+\epsilon)$.
\end{theorem}
\begin{proof}
    By Lemma~\ref{L:covering family exists}, there exists a $(C+1)$-covering family for $K$. Now for $n=1,2,\dots$ in succession,
    if possible, update the family by applying Lemma~\ref{L:reduce covering family for K} with the constant $C$ replaced by $C+2^{-n}$, otherwise make no change. By Lemma~\ref{L:well ordering}, there is a smallest index $n$ beyond which the resulting sequence of families stabilizes. The stabilized value produces a $C$-classifying family for $K$ which is also a $(C+2^{-n})$-covering family for $K$; this proves the claim.
\end{proof}

\section{Nonarchimedean analogue}
\label{sec:nonarchimedean}

Throughout this section, fix a number field $K$.
We next formulate a more flexible setup to remedy the issue raised in \Cref{rem:loxton general case}. The idea is to put the nonarchimedean places of $K$ on an equal footing with the archimedean places, and use a similar strategy to replace polynomials that do not always evaluate to $K$-cyclotomic integers.

\begin{definition}
Let $\mathbf{C} = (C_v)_v$ be a tuple of real numbers indexed by all places of $K$. We say a tuple $\mathbf{C}=(C_v)_v$ is \textit{valid} if $C_v > 0$ for all $v$ and $C_v = 1$ for all but finitely many $v$. For a valid tuple $\mathbf{C}$, we say an element $\alpha \in K^{\mathrm{cycl}}$ is \emph{$\mathbf{C}$-bounded} if, for every place $v$ of $K$ and every place $w$ of $K^{\mathrm{cycl}}$ above $v$, we have $\vert{}\alpha\vert{}_w \leq C_v$.

This definition relates to previous ones as follows. Take $K = \QQ$ and let $\mathbf{C}$ be a tuple with $C_v = c$ for $v$ archimedean and $C_v = 1$ for $v$ nonarchimedean. Then $\alpha \in \QQ^{\mathrm{cycl}}$ is $\mathbf{C}$-bounded if and only if it is a cyclotomic integer with house bounded above by $c$.
\end{definition}

We translate \Cref{def:covering family} into this language as follows.
\begin{definition} \label{def:covering family number field}
For a valid tuple $\mathbf{C}$, a finite set of polynomials $\{P_i(z_1,\dots,z_{n_i})\}_{i \in I}$ with coefficients in $K^{\mathrm{cycl}}$ is a \emph{$\mathbf{C}$-covering family (over $K$)} if every $\mathbf{C}$-bounded element $\alpha \in K^{\mathrm{cycl}}$ can be written as $P_i(\zeta_1,\dots,\zeta_{n_i})$ for some index $i \in I$ and some roots of unity $\zeta_1,\dots,\zeta_{n_i}$.
\end{definition}

We again obtain covering families directly from Loxton's theorem.
\begin{lemma} \label{L:nonarch covering exists}
For any valid tuple $\mathbf{C} = (C_v)_v$, there exists a $\mathbf{C}$-covering family over $K$.
\end{lemma}

\begin{proof}
Let $S_{\infty}$ denote the set of archimedean places of $K$, and let $S_{\mathrm{fin}}$ denote the finite set of non-archimedean places $v$ for which $C_v > 1$. 
Since $S_{\mathrm{fin}}$ is finite, we can choose a nonzero algebraic integer $d \in \mathcal{O}_K$ such that $|d|_v \leq C_v^{-1}$ for all $v \in S_{\mathrm{fin}}$. 

Now let $\alpha \in K^{\mathrm{cycl}}$ be any element satisfying $|\alpha|_w \leq C_v$ for all places $v$ of $K$ and all $w \mid v$. 
Consider the product $\beta = d\alpha$. For any non-archimedean place $v$ of $K$ and any $w \mid v$:
\begin{itemize}
    \item If $v \in S_{\mathrm{fin}}$, we have $|\beta|_w = |d|_w |\alpha|_w \leq C_v^{-1} C_v = 1$.
    \item If $v \notin S_{\mathrm{fin}}$, then $|d|_w \leq 1$ (since $d \in \mathcal{O}_K$) and $|\alpha|_w \leq C_v \leq 1$, so $|\beta|_w \leq 1$.
\end{itemize}
Since $|\beta|_w \leq 1$ for all non-archimedean places $w$, $\beta$ is an algebraic integer in $K^{\mathrm{cycl}}$. Furthermore, at the archimedean places, the house of $\beta$ is strictly bounded:
\[
\house{\beta} \leq \house{d} \cdot \max_{v \in S_{\infty}} C_v := M.
\]

Fix a Loxton function $L$.
By \Cref{loxton+}, there exist a constant $B = B(K)$ and a finite set $E = E(K) \subset K$ such that the algebraic integer $\beta$ can be written as a linear combination of roots of unity:
\[
\beta = \sum_{i=1}^b c_i \zeta_i,
\]
where $c_i \in E$, the $\zeta_i$ are roots of unity, and the length satisfies $b \leq (\# E) \cdot L(B M)$.
Let $N = \lfloor (\# E) \cdot L(B M) \rfloor$. We define the finite family of linear polynomials:
\[
\mathcal{F} = \left\{ P(z_1, \dots, z_b) = \sum_{i=1}^b \frac{c_i}{d} z_i \;\middle|\; 1 \leq b \leq N, \; c_i \in E \right\}.
\]
Since $E \subset K$ and $d \in K$, each polynomial in $\mathcal{F}$ has coefficients in $K \subset K^{\mathrm{cycl}}$. By construction, $\mathcal{F}$ is a finite $\mathbf{C}$-covering family over $K$.
\end{proof}

We next give the translation of \Cref{def:classifying family}, or more precisely its specialization to the case $K = \QQ$.
\begin{definition} \label{def:classifying family number field}
A $\mathbf{C}$-covering family $\{P_i\}_{i \in I}$ is a \emph{$\mathbf{C}$-classifying family} if, conversely, for every index $i \in I$ and every choice of roots of unity $\zeta_1,\dots,\zeta_{n_i}$, $P_i(\zeta_1,\dots,\zeta_{n_i})$ is $\mathbf{C}$-bounded.
\end{definition}

We finally give an analogue of \Cref{L:reduce covering family for K}, with a parallel proof.
\begin{lemma} \label{L: non-archimedean reduction}Let $\mathbf{C} = (C_v)_v$ be a valid tuple, and let $\{P_i\}_{i \in I}$ be a $\mathbf{C}$-covering family over $K$.
%(where each $P_i$ is a polynomial with coefficients in $K^{\mathrm{cycl}}$). 
Suppose that $i \in I$ is an index for which there exist a place $v$ of $K$ and a place $w$ of $K^{\mathrm{cycl}}$ lying above $v$ such that
$$
\sup\left\{ \vert{}\sigma(P_i)(\zeta_1,\dots,\zeta_{n_i})\vert{}_w \colon \sigma\colon K^{\mathrm{cycl}} \hookrightarrow \CC; \zeta_1, \dots, \zeta_{n_i} \in \mu \right\} > C_v,
$$
where $\mu \subset K^{\mathrm{cycl}}$ denotes the group of roots of unity.
\begin{enumerate}
    \item[(a)]
An infinite sequence $\{\alpha_k\}_k$ of distinct tuples $\alpha_k = (\zeta_1^{(k)},\dots,\zeta_{n_i}^{(k)})$ of roots of unity for which $|{P_i(\alpha_k)}|_w \leq C_v$
cannot be strict in $\mathbb{G}_m^{n_i}$.
    \item[(b)]
There exists a finite set of polynomials $\{Q_j(z_1,\dots,z_{n_j})\}_{j \in J}$ with coefficients in $K^{\mathrm{cycl}}$ such that $n_j < n_i$ for each $j \in J$, and$$\{P_{i'}\}_{i' \in I \setminus \{i\}} \bigcup \{Q_j\}_{j \in J}$$is a $\mathbf{C}$-covering family over $K$.
\end{enumerate}
\end{lemma}
\begin{proof}
The proof of (b) remains unchanged from Lemma \ref{L:reduce covering family for K},
as does the proof of (a) in case $v$ is archimedean; we thus focus on proving (a) when $v$ is 
nonarchimedean.
For simplicity we denote the absolute value $|\cdot|_w$ by $|\cdot|$.
We also set the notations  $L, \tilde{P}_i, \zeta_m$ as in the proof of Lemma \ref{L:reduce covering family for K}.

We fix an embedding $\overline{K} \hookrightarrow \CC_v$
and lift the sequence to $\beta_k = (\zeta_m, \alpha_k) \in (\mathbb{P}^{1,{\mathrm{an}}}(\mathbb{C}_v))^{n_i+1}$. Let $ \mathbb{P}^{1,{\mathrm{an}}}(\mathbb{C}_v) \times (\mathbb{P}^{1,{\mathrm{an}}}(\mathbb{C}_v))^{n_i}$ be the ambient compact space. Let $K_k:=K(\beta_k)$ and let $\text{Aut}(K_k)$ denote the set of embeddings of $K_k$ into $\CC_v$. The uniform probability measure supported on the absolute Galois orbit $O(\beta_k)$ is:
$$
\bar{\delta}_{\beta_k} = \frac{1}{[K_k:K]} \sum_{\sigma \in \text{Aut}(K_k)} \delta_{\sigma(\beta_k)}.
$$

Let $\nu_{n_i}$ be the Dirac measure of $(\zeta_G,...,\zeta_G)$  in $(\mathbb{P}^{1,\mathrm{an}}(\CC_v))^{n_i}$ where $\zeta_G$ is the Gauss point of $\mathbb{P}^{1,\mathrm{an}}(\CC_v)$. We define the  probability measure $\nu_W$ as:
$$
\nu_W = \frac{1}{d} \sum_{j=1}^d (\delta_{w_j} \times \nu_{n_i})
$$
which is supported on the disjoint union $W = \bigsqcup_{j=1}^d (\{w_j\} \times \{(\zeta_G,...,\zeta_G)\})$.

We will prove the weak convergence of measures $\bar{\delta}_{\beta_k} \xrightarrow{w} \nu_W$. Write $X=\mathbb{P}^{1,{\mathrm{an}}}(\mathbb{C}_v)$ and $Y=X^{n_i}$, so that $X\times Y$ is a compact and Hausdorff space; it suffices to prove the corresponding convergence of integrals for functions of the form $f(x,y)=g(x)h(y)$ where $g$ is continuous on $X$ and $h$ is continuous on $Y$, as such functions are dense in the set of continuous functions on $X\times Y$. 
We have
\begin{equation} \label{eq:factor}
\begin{split}\int_{X\times Y} g(x)h(y) \, d\bar{\delta}_{\beta_k} &= \frac{1}{[K_k : L] d} \sum_{j=1}^d \sum_{\sigma \in \text{Gal}(\overline{K}/L)} g(\tau_j(\sigma(\zeta_m)))\cdot h(\tau_j(\sigma(\alpha_k)))\\
&=\dfrac{1}{d} \sum_{j=1}^d g(w_j)\left[ \dfrac{1}{[K_k:L]}\sum_{\sigma \in \text{Gal}(\overline{K}/L)}h(\tau_j(\sigma(\alpha_k)))\right].\end{split}
\end{equation}
By \Cref{apply yuan}, we may apply \Cref{yuan_equidistribution}(b) to deduce that the inner term in \eqref{eq:factor}
converges weakly to $\int_Y h(y) d\nu_{n_i}= h(\zeta_G,...,\zeta_G)$ as $k \rightarrow \infty$. We thus have
$$
\lim_{k \rightarrow \infty} \int_{X\times Y} g(x)h(y) \, d\bar{\delta}_{\beta_k}=\dfrac{1}{d} \sum_{j=1}^d g(w_j)h(\zeta_G,...,\zeta_G)=\int_{W} g(x)h(y) d \nu_W
$$ 
and the claim is proven.

Now, fix $t>0$, and choose a continuous function $h_t:\mathbb{R}\to\mathbb{R}$ such that
\[
h_t(x)=
\begin{cases}
1 & x \leq C_v,\\
\in (0,1) & C_v < x < C_v+t,\\
0 & x \geq C_v+t.
\end{cases}
\] 
Define the continuous function $g_{t,i} = h_t \circ |\tilde{P}_i|$ on $(\mathbb{P}^{1,\mathrm{an}}(\mathbb{C}_v))^{n_{i}+1}$.
For any Galois conjugate $Y \in O(\beta_k)$, the condition $|{P_i(\alpha_k)}| \leq C_v$ ensures that $|\tilde{P}_i(Y)| \leq C_v$. This forces $g_{t,i}(Y) = 1$ for every $Y \in O(\beta_k)$, yielding:
$$
\int_G g_{t,i} \, d\bar{\delta}_{\beta_k} = 1.
$$
Conversely, we have arranged that $\max_{Z} |\tilde{P}_i(w_1, Z)| > C_v$ and so $|\tilde{P}_i(w_1, \zeta_G,...,\zeta_G)|>C_v$. Therefore, $g_{t,i}$ is strictly less than $1$ on  the connected component $\{w_1\} \times (\zeta_G,...,\zeta_G)$. Because $\nu_W$ assigns uniform positive mass $1/d$ to this component, we conclude:
$$
\int_G g_{t,i} \, d\nu_W < 1.
$$
This contradicts the weak convergence, completing the proof of (a).
\end{proof}

We now conclude as before to obtain generalizations of 
\Cref{T:general classification} and \Cref{T:semicontinuity general}
(and by extension \Cref{T:general classification rational case} and \Cref{T:semicontinuity rational case}).

\begin{theorem} \label{thm:general classification}
For each place $v$ of $K$, fix a positive real number $C_v$ in such a way that all but finitely many of the $C_v$ are equal to $1$. Then we can find a finite list of polynomials $P_i(z_1,\dots,z_{n_i})$ with coefficients in $K^{\mathrm{cycl}}$ such that for any $\alpha \in K^{\mathrm{cycl}}$, the following two statements are equivalent.
\begin{itemize}
    \item 
    We can write $\alpha = P_i(\zeta_1,\dots,\zeta_{n_i})$ for some index $i$ and some roots of unity $\zeta_1, \dots, \zeta_{n_i}$.
    \item 
    For every place $v$ of $K$ and every place $w$ of $K^{\mathrm{cycl}}$ above $v$,
    $|\alpha|_w \leq C_v$.
\end{itemize}

\end{theorem}
\begin{proof}
    By Lemma~\ref{L:nonarch covering exists}, there exists a finite $\mathbf{C}$-covering family over $K$. We associate to each such covering family $\{P_i\}_{i \in I}$ a sequence of nonnegative integers $(m_0, m_1, \dots)$ with finite support, where each term $m_j$ counts the number of indices $i \in I$ with $n_i = j$.  By Lemma~\ref{L:well ordering}, there is a minimal such sequence; choose a $\mathbf{C}$-covering family realizing this minimum. By minimality, the conclusion of Lemma~\ref{L: non-archimedean reduction} cannot hold for any index $i \in I$. Consequently, the hypothesis of Lemma~\ref{L: non-archimedean reduction} cannot hold for any index $i \in I$. Thus, this minimal $\mathbf{C}$-covering family is simultaneously a $\mathbf{C}$-classifying family.
\end{proof}

\begin{theorem}[Uniform Semicontinuity] \label{T:uniform gap}
Let $K$ be a number field, and let $\mathbf{C} = (C_v)_v$ be a valid tuple. Let $S$ denote the finite set of places of $K$ containing all archimedean places and all finite places where $C_v > 1$. Then there exists a uniform $\epsilon > 0$ such that no element $\alpha \in K^{\mathrm{cycl}}$ simultaneously satisfies $C_{v_0} < \sup_{w_0 \mid v_0} |\alpha|_{w_0} < C_{v_0} + \epsilon$ for some $v_0 \in S$, and $|\alpha|_w < C_v + \epsilon$ for all $w \mid v$ with $v \in S$ (along with $|\alpha|_w \leq 1$ strictly for all $w \mid v$ with $v \notin S$).
\end{theorem}

\begin{proof}
% For each integer $n \geq 1$, define the perturbation $\mathbf{C}^{(n)} = (C_v^{(n)})_v$ where $C_v^{(n)} = C_v + 2^{-n}$ for $v \in S$, and $C_v^{(n)} = 1$ for $v \notin S$. Clearly $\mathbf{C}^{(n)}$ is also a valid tuple. By Lemma~\ref{L:nonarch covering exists}, there exists a finite $\mathbf{C}^{(1)}$-covering family over $K$. 

Now for $n=1,2,\dots$ in succession, if possible, update the family by applying Lemma~\ref{L: non-archimedean reduction} with the tuple of bounds $\mathbf{C}$ replaced by the perturbed tuple $\mathbf{C}^{(n)}$; otherwise, make no change. By Lemma~\ref{L:well ordering}, there is a smallest index $n$ beyond which the resulting sequence of families stabilizes. The stabilized value is then a $\mathbf{C}$-classifying family which is simultaneously a $\mathbf{C}^{(n)}$-covering family; setting $\epsilon = 2^{-n}$ proves the claim.
\end{proof}

\section{Remarks on effectivity}
\label{sec:remarks on effectivity}

As noted in the introduction, we have not formulated any of our results in an effective manner.
We now justify the claim that they can all be made effective, in the following senses.
\begin{itemize}
\item 
In \Cref{T:general classification rational case}, there is an algorithm that, upon input of $C$, 
returns a finite family of polynomials as indicated.
\item
Similarly, in \Cref{thm:general classification}, there is an algorithm that, upon input of $K$ and the $C_v$, returns a finite family of polynomials as indicated.
\item 
In \Cref{T:semicontinuity rational case}, there is an algorithm that, upon input of $C$, returns a suitable value of $\epsilon$ and a proof certificate.
\item
Similarly, in \Cref{T:uniform gap}, there is an algorithm that, upon input of $K$ and the $C_v$, returns a suitable value of $\epsilon$ and a proof certificate.
\end{itemize}

All of these claims boil down to two essential points. The first is that there exists an explicit choice for the Loxton function $L$ in \Cref{T:loxton} (as this then leads to an effective construction in \Cref{loxton+}
as indicated in \cite{DZ}). The proof in \cite{loxton} does not explicitly report such a function, but one can trace through the steps in the proof to obtain one.

%\kiran{add D'Andrea--Narvaez-Clauss--Sombra reference to bibTeX}\jitendra{done}
The second point is that the replacement step in \Cref{L:reduce covering family for K} and \Cref{L: non-archimedean reduction} can be made effective. This can be achieved by replacing Yuan's equidistribution theorem with an effective version for tori, such as \cite{dandrea-narvaezclauss-sombra}.

%---------------------------------------------------------------------
%	  Section -- Acknowledgements
%---------------------------------------------------------------------
\section*{Acknowledgements}

This paper is a result of the virtual workshop ``Rethinking Number Theory 6'' funded by NSF grants DMS-2201085 and DMS-2418528. In addition, Bajpai was supported by the European Research Council (ERC) under the European Union’s Horizon Europe research and innovation
programme (grant agreement No. 101163794, GroupHype). Das was supported by Prime Minister's Research Fellowship, Government of India. Kedlaya was supported by NSF grant DMS-2401536, the UC San Diego Warschawski Professorship, and during fall 2025 by the Lodha Mathematical Sciences Institute. Mello was supported by Oakland  University BFA and FCT grants.

%-----------------------------------------------------------------
%	  Section -- Bibliography
%---------------------------------------------------------------
\nocite{}
\bibliographystyle{abbrv}
\bibliography{Castle}
\end{document}